\documentclass[12pt]{amsart}
\usepackage{graphicx}
\usepackage{color}
\usepackage{amsmath}
\usepackage{amsfonts}
\usepackage{amssymb}
\usepackage{amsthm}
\usepackage{verbatim}
\newcommand{\lap}{\Delta}
\newcommand{\e}{\varepsilon}
\newcommand{\R}{\mathbb{R}}
\newcommand{\K}{\mathbb{K}}
\newcommand{\dd}{\partial}

\newcommand{\dist}{\operatorname{dist}}

\renewcommand{\div}{\operatorname{div}}

\begin{document}
\renewcommand{\rmdefault}{cmr}
\newtheorem{thm}{Theorem}
\newtheorem{lem}[thm]{Lemma}
\newtheorem{cor}[thm]{Corollary}
\newtheorem{prop}[thm]{Proposition}
\newtheorem{definition}[thm]{Definition}
\newtheorem{rem}[thm]{Remark}
\newtheorem{ass}[thm]{Assumption}

\title{Optimal regularity for the obstacle problem for the $p$-Laplacian }

\author{John Andersson}
\address{Department of Mathematics, Royal Institute of Technology, 100 44 Stockholm, Sweden}
\email{johnan@kth.se}
\author{Erik Lindgren}
\email{eriklin@kth.se}
%\address{Department of Mathematics, Royal Institute of Technology, 100 44 Stockholm, Sweden}
\author{Henrik Shahgholian}
\email{henriksh@math.kth.se}
%\address{Department of Mathematics, Royal Institute of Technology, 100 44 Stockholm, Sweden}

\begin{abstract}
In this paper we discuss the obstacle problem for the $p$-Laplace operator. We prove optimal growth results for the solution. Of particular interest is the point-wise regularity of the solution at free boundary points.
The most surprising result we prove is the one for the $p$-obstacle problem: Find the \emph{smallest} $u$ such that 
$$
\hbox{div} (|\nabla u|^{p-2}\nabla u) \leq 0, \qquad u\geq \phi, \qquad \hbox{in } B_1,
$$
with $\phi \in C^{1,1}(B_1)$  and given boundary datum on $\partial B_1$. We prove that the solution is uniformly  $C^{1,1}$ at free boundary points. Similar results are obtained in the case of an inhomogeneity belonging to $L^\infty$.
When applied to the corresponding parabolic problem, these results imply that any solution which is Lipschitz in time is $C^{1,\frac{1}{p-1}}$ in the spatial variables.
\end{abstract}
\maketitle

\section{Introduction}
\subsection{Problem formulation}
In this paper we consider the optimal regularity of minimizers of the constrained $p$-Dirichlet energy 
$$
\int_{B_1} \frac{|\nabla v|^p}{p}+fv\, d x \ , \qquad  v \in \K:= \{w\in W^{1,p}(B_1): \ w \geq \phi, \ w=g \text{ on }\dd B_1 \},
$$ 
where $B_1 \subset \R^n$ ($n\geq 2$) is the unit ball, and
 $\phi$ and $g$ are given functions (in appropriate spaces).
This is  equivalent to  finding the smallest function $u$ such that 
$$
\lap_p u \leq f, \qquad  u\geq \phi \ ,
$$
given the boundary conditions on $\dd B_1$. Here, and in the sequel, $\lap_p u =\div(|\nabla u|^{p-2}\nabla u)$ is the $p$-Laplace operator and $1<p<\infty$.

Of particular interest is the set $\Omega=\{u>\phi\}\cap B_1$ and the free boundary $\Gamma = \partial \{u>\phi\}\cap B_1$. To better understand
the free boundary, $\Gamma$, it is important to first understand the point-wise regularity of the solution $u$. 
In the homogeneous case $f\equiv 0$, we prove, the rather ``unexpected'' result, that the point-wise regularity of $u$ at a free boundary point is the same as the regularity of 
the obstacle $\phi$, at least up to $C^{1,1}$.
That is, if $\phi\in C^{1,1}$ then $u$ leaves $\phi$ in a quadratic fashion. This  surprising result implies, in turn,  that 
the presence of the obstacle actually improves the regularity of the solution to the solution to the $p-$harmonic obstacle problem, at free boundary points.

In the more general and inhomogeneous case we prove that if $\phi\in C^{1,\beta}$ and if $f\in L^\infty(B_1)$ 
then $u$ leaves $\phi$ in $r^{1+\alpha}$-fashion, where
$$
\alpha = \min\left(\frac{1}{p-1},\beta\right).
$$

In the second part of the paper we apply the aforementioned results to the $p$-parabolic obstacle problem. The $p$-parabolic obstacle problem amounts to finding the \emph{smallest} function $u$, defined on $B_1\times (0,T)$, with given boundary and initial data, such that 
$$\left\{ 
\begin{array}{lr}
\lap_p u-\frac{\partial u}{\partial t} \leq 0,\\ u\geq \phi.\end{array} 
\right.
$$
Under the assumption that $\frac{\partial u}{\partial t}\in L^\infty$, we obtain the optimal growth in the spatial variable of order $1+\alpha$, where
$$
\alpha = \min\left(\frac{1}{p-1},\beta\right).
$$
In addition, we show that if the initial datum satisfies $|\lap_p g|\leq C$ and the obstacle and the spatial boundary datum are Lipschitz in time, then so is the solution.

\subsection{Known results}
The non-degenerate elliptic obstacle problem, $p=2$, is very well studied and the regularity properties of the solution are well known. 
It was proved by Frehse in \cite{Fre72} and Kinderlehrer in \cite{Kin70} (in two dimensions) that $u$ is $C^{1,1}$, provided  the same is true for the  obstacle. 
Later in \cite{Caf98} it was proved that the free boundary, except at cusp-like points, is a $C^\infty$ hypersurface. This result was 
sharpened even further in two dimensions by Monneau in \cite{Mon03}. A related but somewhat different problem was studied in 
\cite{KKPS00} and \cite{LS03}. See also \cite{Nor86} and \cite{Lin88} for regularity results relating to the $p-$harmonic obstacle problem. In \cite{KZ02}, the corresponding result of Theorem \ref{thm:quad} is proved, when $C^{1,\beta}$ is replaced by $C^{0,\alpha}$ for some small $\alpha\in (0,1)$, for a more general class of quasilinear operators. The resemblance of the proofs is striking.

For the parabolic obstacle problem, there is a series of papers  \cite{BDM05}, \cite{BDM06}, \cite{Bla06a} and \cite{Bla06b}, where optimal regularity as well as the regularity of the free boundary is proved in the case when $p=2$, for variable coefficents and variable right-hand side. In the papers \cite{LM13a} and \cite{LM13b}, the right-hand side is allowed to be merely in $L^p$. In \cite{PS07}, the elliptic part of the operator is allowed to be fully nonlinear.  A slightly more general free boundary problem of parabolic type is studied in \cite{CPS04}, \cite{EP08}, \cite{EL12} and \cite{ALS13}.

In the $p-$parabolic case we refer the reader to the literature: \cite{Choe}, \cite{KKS09}, \cite{Lin12}, \cite{LP12} and \cite{KMN12}. 
One of the authors studied a quite similar problem in \cite{Sha03}.

\subsection{Main idea} Roughly speaking, the main idea for the elliptic problem is the following: When the gradient is large then the equation is 
non-degenerate and classical estimates apply. But on the other hand, when the gradient is small we can rescale and obtain 
uniform bounds using the weak Harnack inequality (which applies to supersolutions).

\subsection{Acknowledgement}
We thank Peter Lindqvist for several  encouraging and useful comments, after reading the manuscript at an early stage. The second author is supported by the Swedish Research Council, grant no. 2012-3124. The third author is partially supported by the Swedish Research Council.

%%%%%%%%%%%%%%%%%%%%%%%%%%%%%%%%%%%%%%%%%%%%
\section{The elliptic problem}
%%%%%%%%%%%%%%%%%%%%%%%%%%%%%%%%%%%%%%%%%%%%

In this part we treat the elliptic problem. Given an open, bounded set $\Omega$ and some boundary datum given by the restriction of 
$g\in W^{1,p}(\Omega)$ to  $\partial \Omega$, we say that $u$ is a solution of the $p$-obstacle problem in $\Omega$ with obstacle $\phi\in C^{1,\beta}$, $g\ge \phi$, and with inhomogeneity $f\in L^\infty(B_1)$,
if $u$ minimizes
$$
\int_{\Omega}\frac{|\nabla u|^p}{p} +fu \, dx
$$
subject to $u\geq\phi$ in $\Omega$ and $u=g$ on $\partial \Omega$. 

The main result of this paper is the optimal growth at free boundary points.
\begin{thm}\label{thm:quad} Let $p\in (1,\infty)$, $\beta\in (0,1]$ and $u$ be a solution to the $p$-obstacle problem in $B_1$ with obstacle $\phi\in C^{1,\beta}(B_1)$ and $f\in L^{\infty}(B_1)$. Suppose further  
$$
 \|\phi\|_{C^{1,\beta}(B_1)}\leq N,\quad \|f\|_{L^\infty(B_1)}\leq L.
$$
Then for any point $y\in \Gamma\cap B_{1/2}$ and for $r<1/2$ 
\begin{equation}\label{ElliptEst}
\sup_{x\in B_r(y)}|u(x)-u(y)-(x-y)\cdot \nabla u(y)|\leq C(N^{p-1}+L)^\frac{1}{p-1} r^{1+\alpha},
\end{equation}
where $C=C(\beta,p)$ and
$$
\alpha = \min\left( \frac{1}{p-1},\beta \right),
$$
and  $\alpha =\beta$ if $f\equiv 0$. In particular, 
$$
\sup_{x\in B_r(y)}|u(x)-\phi(x)|\leq (C+1)(N^{p-1}+L)^\frac{1}{p-1}r^{1+\alpha}.
$$
\end{thm}
\begin{proof} We give the proof in the case $f\not\equiv 0$. The only difference in proving the case $f\equiv 0$ would be the scaling of order $\beta$. By simply considering the normalized function 
$$
\frac{u}{(N^{p-1}+L)^\frac{1}{p-1}},
$$
we can assume that $u$ solves the $p$-obstacle problem with obstacle $\phi$ satisfying $\|\phi\|_{C^{1,\beta}(B_1)}\leq 1/2$ and with $\| f\|_{L^\infty(B_1)}\leq 1$. Then at any free boundary point $y$, we have $|\nabla u(y)|=|\nabla \phi(y)|\leq 1/2$. The proof is now divided into different cases. The correctly 
scaled estimate is then obtained in the end by multiplying the constant with $(N^{p-1}+L)^\frac{1}{p-1}$.\\ 

\noindent {\bf Case 1: When $|\nabla u(y)|< r^\alpha<(1/2)^\alpha$:} When $|\nabla u(y)|<r^\alpha$ it follows from the triangle inequality that 
$$
\sup_{x\in B_r(y)}|u(x)-u(y)|\leq Cr^{1+\alpha}
$$
implies
$$
 \sup_{x\in B_r(y)}|u(x)-u(y)-(x-y)\cdot \nabla u(y)|\leq (C+1)r^{1+\alpha}.
$$
It is therefore enough to prove 
$$
\sup_{B_r(y)} |u(x)-u(y)|\leq Cr^{1+\alpha},
$$
for some constant $C=C(\beta,p)$. To this end we define the rescaled functions
$$
\tilde \phi(x)=\frac{\phi(rx+y)-\phi(y)}{r^{1+\alpha}}
$$
and
$$
\tilde u (x)=\frac{u(rx+y)-u(y)}{r^{1+\alpha}}.
$$
We may estimate the $L^\infty$ norm of $\tilde{\phi}$ according to
$$
\|\tilde \phi\|_{L^\infty(B_1)}=\left\|\frac{\phi(rx+y)-\phi(y)}{r^{1+\alpha}}\right\|_{L^\infty(B_1)}\le 
$$
$$
\le \left\|\frac{ \phi(rx+y)-\phi(y)-r\nabla \phi(y)\cdot (x-y)}{r^{1+\alpha}}\right\|_{L^\infty(B_1)}+\left\| \frac{\nabla \phi(y)\cdot (x-y)}{r^{\alpha}}\right\|_{L^\infty(B_1)}\leq 3/2,
$$ 
where we used Proposition \ref{prop:standard} in the appendix to conclude that
$$
|\nabla \phi(y)| =|\nabla u(y)| \leq r^\alpha.
$$
We note that $\tilde u$ minimizes
$$
\int_{B_1}\frac{|\nabla \tilde u|^p}{p}+\tilde f\tilde u \,d x
$$
with $\tilde \phi$ as obstacle and where 
$$
\tilde f(x)=r^{1-\alpha(p-1)}f(rx),
$$
so that $\|\tilde f\|_{L^\infty(B_1)}\leq 1$. Thus,
 $$\lap_p \tilde u\leq \tilde f,\quad \text{ in $B_1$}.
 $$
 The weak Harnack inequality, see for instance Theorem 3.13 in \cite{MZ97}, 
applied to the non-negative function $\tilde u+3/2\ge -\| \tilde{\phi}\|_{L^\infty}+3/2\ge 0$ implies
$$
\|\tilde u+3/2\|_{L^s(B_\frac34)}\leq C_1(p)\inf_{B_\frac12}(\tilde u+3/2)\leq C_1(p)\left(\tilde \phi(0)+3/2\right)\leq  4C_1(p)=C_2(p).
$$
for some $s>1$. Now, let $$v=\max(\tilde u+3/2,\sup_{B_1}\tilde \phi+3/2).$$ Then $\lap_p v\geq \tilde f$\footnote{Observe that $\lap_p v=\lap_p (\max(\tilde u -\sup\tilde  \phi, 0) + 3/2 + \sup \tilde \phi)=\tilde f$ and $\lap_p (\tilde u -\sup\tilde  \phi)=\tilde f$ in the set $\{\tilde u -\sup\tilde  \phi >0\}$ and it is zero outside this set. By taking a test function of the form $\phi\eta((\tilde u -\sup\tilde  \phi)^+)$, with $\phi\in C_0^\infty(B_1)$ and $\eta$ a linear approximation of the identity, it follows that $\lap_p v\geq \tilde f$.} and thus from the sup-estimate for subsolutions (cf. Corollary 3.10 in \cite{MZ97}) we can conclude together with the estimate above that
$$
\sup_{B_\frac12} v\leq C_3(p)\|v\|_{L^s(B_\frac34)}\leq C_3(p)C_2(p).
$$
This implies, upon relabeling the constants, that 
\begin{equation}\label{eq:supestutilde}
\sup_{B_\frac12} \tilde u\leq C(p) .
\end{equation}
Since moreover $\tilde u\geq \tilde \phi \geq -3/2$, $\tilde u$ is uniformly bounded in $L^\infty(B_{1/2})$, which implies the desired estimate, for $r<1/4$. In order to obtain the estimate for $r\in (1/4,1/2)$ one just needs to increase the constant by $2^{1+\alpha}$.

%%%%%%%%%%%%%%%%%%%%%%%%%%%%%%%%%%%%%%%%%%%%%%%
%%%%%%%%%%%%%%%%%%%%%%%%%%%%%%%%%%%%%%%%%%%%%%%

\noindent {\bf Case 2: When $|\nabla u(y)|\geq r^\alpha$, $r<1/2$:} From Case 1 we know that 
\begin{equation}\label{eq:sry}
\sup_{B_{r_y}(y)} u(x)\leq C(p) r_y^{1+\alpha}
\end{equation}
where $r_y^\alpha=|\nabla u(y)|$. Let 
$$
\tilde \phi(x)=\frac{\phi(r_yx+y)-\phi(y)}{r_y^{1+\alpha}}.
$$
Then 
\begin{equation}\label{eq:gradeq1}
|\nabla \tilde \phi(0)|=|\nabla \tilde u(0)|=1.
\end{equation}
Define also
$$\tilde u(x)=\frac{u(r_yx+y)-u(y)}{r_y^{1+\alpha}}
$$
and
$$
\tilde f(x)=r_y^{1-\alpha(p-1)}f(r_y x).
$$
Then $\tilde u$ solves the $p$-obstacle problem in $B_1$ with $\tilde \phi$ as an obstacle and with inhomogeneity $\tilde f$, satisfying $\|\tilde f\|_{L^\infty(B_1)}\leq 1$. Moreover, from the assumption $\|\phi\|_{C^{1,\beta}}\le 1/2$, 
$$
\|\tilde \phi\|_{C^{1,\beta}(B_{1/2})}\leq C_4,
$$
and by \eqref{eq:sry} $\tilde u$ is uniformly bounded in $L^\infty(B_{1/2})$. From Proposition \ref{prop:standard} it follows that 
$$
\|\tilde u\|_{C^{1,\tau}(B_\frac12)}\leq C_5, \quad \tau =\tau(p), \quad C_5=C_5(p).
$$
This together with \eqref{eq:gradeq1} implies that we can find $r_0=r_0(p)$ so that $|\nabla \tilde u|\geq 1/2$ in $B_{r_0}$. Hence, $\tilde u$ is a uniformly bounded solution to the obstacle problem for a uniformly elliptic operator with $C^\tau$-coefficients in $B_{r_0}$, with a  $C^{1,\beta}$ regular obstacle and uniformly bounded inhomogeneity. From Proposition \ref{prop:standard} and Proposition \ref{prop:standard2}
$$
\|\tilde u\|_{C^{1,\alpha}(B_{r_0})}\leq C(p,\beta).
$$
Scaling back we obtain
$$
\| u\|_{C^{1,\alpha}(B_{r_0r_y}(y))}\leq C(p,\beta), 
$$
which in particular implies
$$
\sup_{B_r(y)}|u(x)-u(y)-(x-y)\cdot \nabla u(y)|\leq C r^{1+\alpha}, 
$$
for $r<r_0r_y=r_0|\nabla u(y)|^\frac{1}{\alpha}$.

\vspace{3mm}

\noindent {\bf Conclusion:} In both Case 1 and Case 2 we concluded that  
$$
\sup_{B_r(y)}|u(x)-u(y)-x\cdot\nabla u(y)|\leq C r^{1+\alpha}, 
$$
for all $r<1/2$ such that $r\le r_0|\nabla u(y)|^\frac{1}{\alpha}$ (Case 2) and when 
$r>|\nabla u(y)|^\frac{1}{\alpha}$ (Case 1). Therefore we only need to fill the gap when  
\begin{equation}\label{TheGap}
r_0|\nabla u(y)|^\frac{1}{\alpha}< r < |\nabla u(y)|^\frac{1}{\alpha}.
\end{equation}
Assume that $r$ is in the interval specified in (\ref{TheGap}). Then 
$$
\sup_{B_r(y)}|u(x)-u(y)-(x-y)\cdot\nabla u(y)|\leq \sup_{B_{r_y}(y)}|u(x)-u(y)-(x-y)\cdot\nabla u(y)|.$$
Hence, 
\begin{align*}\sup_{B_r(y)}|u(x)-u(y)-(x-y)\cdot\nabla u(y)|&\leq\sup_{B_{r_y}(y)}|u(x)-u(y)-(x-y)\cdot\nabla u(y)|\\&\leq C r_y^{1+\alpha}= C\frac{r_y^{1+\alpha}}{r^{1+\alpha}}r^{1+\alpha}
\leq \frac{C}{r_0^{1+\alpha}}r^{1+\alpha}.
\end{align*}
We thus have the estimate for all $r<1/2$. To obtain the estimate for the original $u$ (not rescaled by a factor $(N^{p-1}+L)^\frac{1}{p-1}$) one just needs to multiply the constant $C$ with $(N^{p-1}+L)^\frac{1}{p-1}$.

The last estimate follows from 
\begin{align*}
&\sup_{x\in B_r(y)}|u(x)-\phi(x)|\\
&=\sup_{x\in B_r(y)}|u(x)-u(y)-(x-y)\cdot \nabla u(y)+\phi(y)+(x-y)\cdot \nabla \phi(y)-\phi(x)|\\
&\leq  (C+1) r^{1+\alpha}.
\end{align*}
\end{proof}

\section{Non-degeneracy and porosity of the free boundary in the homogeneous case}
In this section we treat the homogeneous case in more detail, assuming also that $p>2$. We prove by standard arguments that the difference $u-\phi$ cannot decay faster than quadratic around free boundary points. 
This combined with the optimal quadratic growth implies, by a standard argument, that the free boundary $\Gamma$ is porous. We recall that $\Gamma\cap B_{1/2}$ is said to be \emph{porous} if there exists a $\delta>0$
such that for every $y\in \Gamma\cap B_{1/2}$ and $r\in (0,1/4)$
$$
\frac{|\Gamma\cap B_r(y)|}{|B_r(y)|}\le 1-\delta.
$$
Since this directly implies that $\Gamma$ has no Lebesgue points it follows that the free boundary has measure zero. The notion of porosity was introduced in \cite{Dol67}; See also the survey \cite{Zaj87}.  
%\ggreen{I think the argument below should work without $\phi \in C^2$. Needs more work thoug. Worth trying!}
\begin{prop}\label{prop:nondeg} Let $p\in (2,\infty)$ 
%\bblue{What is the reason for the restriction of $p$ here? not quite sure I can go through with the proof otherwise, $\lap_p v$ cont wrt $\e$?} 
and let $u$ be a solution to the $p$-obstacle problem in $B_1$ with obstacle $\phi\in C^2(B_1)$ with $f\equiv 0$. Suppose further that $\lap_p \phi <0$. Then there is a 
constant $\e=\e (\sup \lap_p \phi)$ such that for any $x^0\in \Gamma$ and $r<\dist(x^0,\partial B_1)$ there holds
$$
\sup_{\dd B_r(x^0)\cap \{u>\phi\}} (u-\phi)\geq \e r^2.
$$
\end{prop}
\begin{proof} The proof is standard. Take $y\in \{u>\phi\}$. Let $v(x)=\phi(x)+\e |x-y|^2$, where $\e$ is chosen  small enough such that  $\lap_p v<0$. This is indeed possible since $\lap_p v$ is continuous with respect to $\e$. Pick $r<\dist(x^0,\partial B_1)$. Then $\lap_p u=0\geq \lap_p v$ in $\{u>\phi\}\cap B_r(x^0)$. Moreover, $u(y)\geq \phi(y)=v(y)$. From the comparison  principle it follows that there is $z_y\in \dd\left(\{u>\phi\}\cap B_r(x^0)\right)$ such that $u(z_y)\geq v(z_y)$. Since $u <  v$ on $B_r(x^0)\cap \dd\{u>\phi\}$ there must be $z_y\in \{u>\phi\}\cap \dd B_r(x^0)$ such that $u(z_y)\geq v(z_y)$. The result follows by continuity and by letting $y\to x^0$.
\end{proof}
\begin{rem} The only place where we need to impose the condition $p>2$ is in the proof above. The proof requires that
$$
\lap_p (\phi(x)+\e |x-y|^2)
$$
is continuous with respect to $\e$. This is not necessarily true when $p<2$. However, we believe that the result holds true even in the case $p<2$, and that the assumption is merely an artifact of the proof.
\end{rem}
%%%%%%% THE OLD GRADIENT ESTIMATE LEMMA THAT I NOW MOVED TO PREVIOUS SECTION AS A COR %%%%%%%%
\begin{comment}
In order to prove that the free boundary is porous we also need the following gradient estimate.

\end{comment}
We are now ready to give the proof of porosity.
\begin{cor}\label{cor:por} Under the assumptions in Proposition \ref{prop:nondeg}, the free boundary is porous. In particular it has Lebesgue measure zero.
\end{cor}
\begin{proof}
%%%%%%%%%%% OLD PROOF THAT USED GRADIENT ESTIMATE %%%%%%%
\begin{comment} Take $x^0\in \Gamma$. By Proposition \ref{prop:nondeg}, for $r$ small enough, there is  $y\in \dd B_r(x^0)$ such that  
$$
u(y)-\phi(y)\geq \e r^2.
$$
Now take $\rho<1$ so that $B_{\rho r}(y)\subset B_{2r}(x^0)$. Lemma \ref{lem:gradest} implies that $|\nabla (u-\phi)|\leq Cr$ in $B_{\rho r}(y)$. Thus, for $z\in B_{\rho r}(y)$
$$
u(z)-\phi(z)\geq u(y)-\phi(y)-|z-y|\sup_{B_{\rho r}(y)}|\nabla u|\geq \e r^2-C'\rho r^2=r^2(\e-C\rho)>0
$$
whenever $\rho$ is small enough. Hence, $B_{\rho r}(y)\subset \{u>\phi\}$ for $\rho$ small enough, which means exactly that $\Gamma$ is porous. From Lebesgue's density theorem it follows that $\Gamma$ has zero Lebesgue density.
\end{comment}
At a closer look at Theorem \ref{thm:quad}, we see that in fact we prove
$$
|u(x)-\phi(x)|\leq CN (\dist(x,\Gamma))^2
$$
for $x\in B_\frac12$, given that $\phi$ is $C^{1,1}$. Now, pick $x^0\in \Gamma$. By Proposition \ref{prop:nondeg}, for $r$ small enough, there is  $y\in \dd B_r(x^0)$ such that  
$$
u(y)-\phi(y)\geq \e r^2.
$$
Combining the above estimates we arrive at
$$
\e r^2\leq CN (\dist(y,\Gamma)^2,
$$
so that 
$$\dist(y,\Gamma)>\left(\frac{\e}{CN}\right)^\frac12r:=\delta r. $$
Hence, 
$$
\Gamma\cap B_{\delta r}(y)=\emptyset.
$$
Hence, we can find a point $z\in B_{\delta r}(y)$ so that 
$$B_{\frac{\delta r}{2}}(z)\subset B_{\delta r}(y)\cap B_r(x^0),$$
which implies
$$
\frac{|\Gamma\cap B_r(x^0)|}{|B_r(x^0)|}\leq 1-\left(\frac{\delta}{2}\right)^n, 
$$
which means exactly that $\Gamma$ is porous. From Lebesgue's density theorem it follows that $\Gamma$ has zero Lebesgue density.
\end{proof}

\section{Application to the parabolic problem}
%%%%%%%%%%%%%%%%%%%%%%%%%%%%%%%%%%%%%%%%%%%%
In this part we mention an application of the previous results to the parabolic problem introduced earlier. Throughout this section, $q=\frac{p}{p-1}$. We introduce the notation
$$
Q_r^-(x,t)=B_r(x,t)\times (-r^q+t,t],\quad \partial_p Q_r^-(x,t)=\dd B_r(x,t)\times (-r^q+t,t]\cup B_r(x,t)\times \{t\},$$
with the simplification $Q_r^-=Q_r^-(0,0)$. Given boundary datum $g$ on 
$\partial_p Q_r^-$, we say that $u$ is a solution of the $p$-parabolic obstacle problem in 
$Q_r^-$ with obstacle $\phi$ if $u$ satisfies
\begin{equation}\left\{
\begin{array}{lr}
\max (\lap_p u-u_t , u-\phi)=0\text{ in $Q_r^-$} ,\\
u=g \text{ on $\partial_p Q_r^-$}.
\end{array}\right.
\end{equation}
As a simple corollary of Theorem \ref{thm:quad} we obtain that if a solution is Lipschitz in time, then it has the optimal growth of order 
$q$ in the spatial variables, at free boundary points. We also give an example of assumptions under which the solution is Lipschitz in time. The main result of this section is stated below:
\begin{thm}\label{thm:medhenrik} Let $p\in (1,\infty)$ and let $u$ be a solution to the $p$-parabolic obstacle problem in $Q_1^-$ with obstacle $\phi\in C^2(Q_1^-)$. Suppose further that 
$$
|u_t|\leq L,\quad  \|\phi\|_{C^2(Q_1^-)}\leq N.
$$
Then for any point $(y,s)\in \Gamma\cap Q_{1/2}^-$ and for $r<1/4$ 
\begin{equation}\label{parabolicEst}
\sup_{(x,t)\in Q_r^-(y,s)}|u(x,t)-u(y,s)-(x-y)\cdot \nabla u(y,s)|\leq Cr^{q},\quad q=\frac{p}{p-1}.
\end{equation}
where $C=C(p,L,N)$.
\end{thm} 

The assumption that the solution is Lipschitz in time in Theorem \ref{thm:medhenrik} is rather unsatisfactory since we do not know if it
is true in general even for solutions of the equation $u_t=\lap_p u$, without the presence of an obstacle. However, if the initial datum $g$ satisfies $|\lap_p g|\leq C$ and the obstacle and the boundary datum are Lipschitz in time, then so is the solution, as is shown below. In 
the following lemma we will, for notational convenience, assume that the solution is defined in $Q^+_1$ instead of $Q_1^-$.

\begin{lem}\label{lem:ut} Let $p\in (1,\infty)$. Assume that $u$ is a solution to the $p$-parabolic obstacle problem in $Q_1^+$ with obstacle 
$\phi$ and spatial boundary datum $f$ and initial datum $g$. Suppose further that $|f_t|\leq N$, $|\phi_t|\leq N$ in $Q_1^+$ and $|\lap_p g|\leq N$, for some constant $N>0$. Then
$$
|u_t|\leq  N.
$$
\end{lem}
\begin{proof} The function $g(x)+Nt$ is a supersolution of the equation. Moreover, $g(x)+Nt\geq u(x,t)$ on $\partial B_1$ and for $t=0$. In addition, since $g(x)\geq \phi(x,0)$ and $|\phi_t|\leq N$, also $g(x)+Nt\geq \phi(x,0)+Nt\geq \phi(x,t)$. Hence, $g(x)+Nt\geq u(x,t)$ in $Q_1^+$. Now, let $w(x,t)=u(x,t+h)$ for some $h>0$. Then $w(x,t)=f(x,t+h)\leq f(x,t)+Nh$ for $x\in \partial B_1$ and $w(x,0)=u(x,h)\leq g(x)+Nh$. Moreover, $w(x,t)\geq \phi(x,t+h)$, where $\phi(x,t+h)\leq \phi(x,t)+Nh$. Hence $w$ is less than or equal to the solution of the $p$-parabolic obstacle problem in $Q_1^+$ with spatial boundary datum $f(x,t)+Nh$, initial datum  $g(x)+Nh$ and obstacle $\phi(x,t)+Nh$. This solution is of course given by $u(x,t)+Nh$ so that we obtain $u(x,t+h)=w(x,t)\leq u(x,t)+Nh$. By similar arguments, $u(x,t+h)\geq u(x,t)-Nh$. 
\end{proof}

The corollary below is immediate.
\begin{cor} Let $p\in (1,\infty)$. Assume the hypotheses of Lemma \ref{lem:ut} and that $\phi\in C^2(Q_1^-)$. Then for any point $(y,s)\in \Gamma\cap Q_{1/2}^-$ and for $r<1/4$
$$
\sup_{(x,t)\in Q^-_r(y,s)}|u(x,t)-u(y,s)-(x-y)\cdot \nabla u(y,s)|\leq Cr^{q},\quad q=\frac{p}{p-1}.
$$
where $C=C(\displaystyle p,N,\|\phi\|_{C^2(Q_1^-)})$.
\end{cor}

%%%%%%%%%%%%%%%%%%%%%%%%%%%%%%%%%%%%%%%%%%%%%%%%%%%%%%%%%%%%%%%%%%%%%%%%%%%%%%%%%%%%%%%
%%%%%%%%%%%%%%%%%%%%%%%%%%%%%%%%%%%%%%%%%%%%%%%%%%%%%%%%%%%%%%%%%%%%%%%%%%%%%%%%%%%%%%%
%%%%%%%%%%%%%%%%%%%%%%%%%%%%%%%%%%%%%%%%%%%%%%%%%%%%%%%%%%%%%%%%%%%%%%%%%%%%%%%%%%%%%%% Here 

\begin{rem}
 It seems plausible that without any boundedness condition on $D_t u$, one should be able to deduce an optimal growth for $u$, of order $q=p/(p-1)$, from the free boundary.
An argument, used by one of the authors (see \cite{Sha03}), and based on scaling and blow-up technique was done for the Stefan problem (with the assumption $D_t u \geq 0$).
Although we could not reverify the inequality (2.15) therein, and what follows,  of the main result in \cite{Sha03}, we still believe that the statement should be true even under weaker assumptions (with no bound on $D_t u$). 

In the scaling argument of \cite{Sha03}, one ends up with a global solution $u_0$ for $t<0$. The solution also behaves like $R^q$ for large $R$. Meanwhile we also have $u_0\geq 0$, and $u_0(0,0)=0$. With the extra condition $D_t u_0 \geq 0$, and hence $u_0(0,t)=0$ for $t<0$. It is also not hard to conclude that $u_0$ is a solution to the $p$-parabolic equation;
indeed any solution to the homogeneous obstacle problem with zero obstacle is a solution to the equation itself. 
One can use these properties along with intrinsic Harnack's inequality to conclude a behavior of type $c(|x|^p/-t)^{1/(p-2)}$ for $t<0$. The question, that we were not able to answer, is whether such a non-trivial solution exists? The answer we conjecture is that no, there is no such non-trivial solution.

It is noteworthy that the function 
$$U=c_p(|x|^p/-t)^{1/(p-2)}, \qquad  t<0,$$
 and with  $c_p=(p-2)\left(\frac{p}{p-2}\right)^{p-1}(\frac{p}{p-2} + n)^{\frac{1}{2-p}}$ is a solution to the equation. Another example is that of Barenblatt 
$${\mathcal B}_p= t^{-n/\lambda}\left(c - \frac{p-2}{p}\lambda^{-1/(p-1)}\left(
\frac{|x|}{t^{1/\lambda}} \right)^{p/(p-1)}  \right)_+^{(p-1)/(p-2)} .
$$
Here one can see two types of behavior for the Barenblatt solution, one is  when the solution touches zero-obstacle, where we see an order of $(p-1)/(p-2)$ 
and the next is along the $t$-axis, with order $p/(p-1)$.

Another example with order $p/(p-1)$ can be given by 
$$
u=c_p(x_+)^{p/(p-1)}+t ,\qquad c_p=(p/(p-1))^{p-1}, \qquad 1<p<\infty, 
$$
where the obstacle is  $\phi = t$.

Yet another example with zero obstacle is
$$
u= A\left(1-x+ct\right)_+^{\frac{p-1}{p-2}}, \qquad A=c^{\frac{1}{p-1}}\left( \frac{p-2}{p-1} \right)^{\frac{p-1}{p-2}}.
$$

In the examples above we see three different behaviors
$$
   r^{\frac{p}{p-1}}, \qquad r^{\frac{p-1}{p-2}}, \qquad r^{\frac{p}{p-2}}.
$$

It is tantalizing to find out how many different growth rates that can be found for solutions, and what are the largest and smallest rates.
\end{rem}

\section{Appendix}
In this section, we recall some well known facts. The proposition below states that if the obstacle is in $C^{1,\beta}(B_1)$ and the inhomogeneity is in $L^\infty(B_1)$, then any bounded solution to the $p$-obstacle problem with is locally in $C^{1,\alpha}$ for some $\alpha$, see \cite{Nor86} and \cite{Lin88}.
\begin{prop}\label{prop:standard} Let $p\in (1,\infty)$ and $u$ solve the $p$-obstacle problem in $B_1$ with $\phi\in C^{1,\beta}(B_1)$ as obstacle and $f\in L^\infty(B_1)$ as inhomogeneity. Then there is $\alpha(p,\|u\|_{L^\infty(B_1)},\|\phi\|_{C^{1,\beta}(B_1)},\|f\|_{L^\infty(B_1)})$ such that 

$$
\|u\|_{C^{1,\alpha}(B_\frac12)}\leq C(p,\|u\|_{L^\infty(B_1)},\|\phi\|_{C^{1,\beta}(B_1)},\|f\|_{L^\infty(B_1)}).
$$  
\end{prop}
It is also well known that the solution to the obstacle problem for a uniformly elliptic operator with $C^\alpha$ coefficients leaves the obstacle in a $r^{1+\gamma}$-fashion if the obstacle is $C^{1,\gamma}$-regular. This follows for instance from Corollary 2.6 in \cite{CK80}.
\begin{prop}\label{prop:standard2} Let $u$ be a solution of the following obstacle problem: The smallest $u$ such that
$$\left\{ 
\begin{array}{lr}
\div({\bf A} \nabla u) \leq f\\ u\geq \phi\end{array}\quad  \hbox{in } B_1
\right.
$$ 
where $\phi\in C^{1,\gamma}(B_1)$ with $\gamma\in (0,1)$, ${\bf A}\in C^\alpha(B_1)$, $f\in L^\infty(B_1)$ and 
$$
\lambda|\xi|^2 \leq {\bf A}\xi\cdot \xi\leq \Lambda|\xi|^2,\quad  0<\lambda<\Lambda
$$
for all $\xi\in \R^n$. Then for $y\in \partial \{u>\phi\}$ and $r<1/2$
$$
\sup_{B_r(y)}|u(x)-u(y)-(x-y)\cdot \nabla \phi(y)|\leq Cr^{1+\gamma},
$$  
where $C=C(p,\|u\|_{L^\infty(B_1)},\|\phi\|_{C^{1,\gamma}(B_1)},\|{\bf A}\|_{C^\alpha(B_1)},\|f\|_{L^\infty(B_1)},\Lambda,\lambda,\gamma)$. In the homogeneous case $f\equiv 0$ we are allowed to choose $\gamma=1$.
\end{prop}

\bibliographystyle{plain}
\bibliography{ref.bib}

\begin{thebibliography}{10}

\bibitem{ALS13}
John Andersson, Erik Lindgren, and Henrik Shahgholian.
\newblock Optimal regularity for the parabolic no-sign obstacle type problem.
\newblock {\em Interfaces Free Bound.}, 15(4):477--499, 2013.

\bibitem{Bla06b}
Adrien Blanchet.
\newblock On the regularity of the free boundary in the parabolic obstacle
  problem. {A}pplication to {A}merican options.
\newblock {\em Nonlinear Anal.}, 65(7):1362--1378, 2006.

\bibitem{Bla06a}
Adrien Blanchet.
\newblock On the singular set of the parabolic obstacle problem.
\newblock {\em J. Differential Equations}, 231(2):656--672, 2006.

\bibitem{BDM05}
Adrien Blanchet, Jean Dolbeault, and R{\'e}gis Monneau.
\newblock On the one-dimensional parabolic obstacle problem with variable
  coefficients.
\newblock In {\em Elliptic and parabolic problems}, volume~63 of {\em Progr.
  Nonlinear Differential Equations Appl.}, pages 59--66. Birkh\"auser, Basel,
  2005.

\bibitem{BDM06}
Adrien Blanchet, Jean Dolbeault, and R{\'e}gis Monneau.
\newblock On the continuity of the time derivative of the solution to the
  parabolic obstacle problem with variable coefficients.
\newblock {\em J. Math. Pures Appl. (9)}, 85(3):371--414, 2006.

\bibitem{CPS04}
Luis Caffarelli, Arshak Petrosyan, and Henrik Shahgholian.
\newblock Regularity of a free boundary in parabolic potential theory.
\newblock {\em J. Amer. Math. Soc.}, 17(4):827--869, 2004.

\bibitem{Caf98}
Luis~A. Caffarelli.
\newblock The obstacle problem revisited.
\newblock {\em J. Fourier Anal. Appl.}, 4(4-5):383--402, 1998.

\bibitem{CK80}
Luis~A. Caffarelli and David Kinderlehrer.
\newblock Potential methods in variational inequalities.
\newblock {\em J. Analyse Math.}, 37:285--295, 1980.

\bibitem{Choe}
Hi~Jun Choe.
\newblock A regularity theory for a more general class of quasilinear parabolic
  partial differential equations and variational inequalities.
\newblock {\em Differential Integral Equations}, 5(4):915--944, 1992.

\bibitem{Dol67}
Evgenii~Prokofevich Dol{\v{z}}enko.
\newblock Boundary properties of arbitrary functions.
\newblock {\em Izv. Akad. Nauk SSSR Ser. Mat.}, 31:3--14, 1967.

\bibitem{EL12}
Anders Edquist and Erik Lindgren.
\newblock Regularity of a parabolic free boundary problem with {H}\"older
  continuous coefficients.
\newblock {\em Comm. Partial Differential Equations}, 37(7):1161--1185, 2012.

\bibitem{EP08}
Anders Edquist and Arshak Petrosyan.
\newblock A parabolic almost monotonicity formula.
\newblock {\em Math. Ann.}, 341(2):429--454, 2008.

\bibitem{Fre72}
Jens Frehse.
\newblock On the regularity of the solution of a second order variational
  inequality.
\newblock {\em Boll. Un. Mat. Ital. (4)}, 6:312--315, 1972.

\bibitem{KKPS00}
Lavi Karp, Tero Kilpel{\"a}inen, Arshak Petrosyan, and H.~Shahgholian.
\newblock On the porosity of free boundaries in degenerate variational
  inequalities.
\newblock {\em J. Differential Equations}, 164(1):110--117, 2000.

\bibitem{KZ02}
Tero Kilpel{\"a}inen and Xiao Zhong.
\newblock Removable sets for continuous solutions of quasilinear elliptic
  equations.
\newblock {\em Proc. Amer. Math. Soc.}, 130(6):1681--1688 (electronic), 2002.

\bibitem{Kin70}
David Kinderlehrer.
\newblock The coincidence set of solutions of certain variational inequalities.
\newblock {\em Arch. Rational Mech. Anal.}, 40:231--250, 1970/1971.

\bibitem{KKS09}
Riikka Korte, Tuomo Kuusi, and Juhana Siljander.
\newblock Obstacle problem for nonlinear parabolic equations.
\newblock {\em J. Differential Equations}, 246(9):3668--3680, 2009.

\bibitem{KMN12}
Tuomo Kuusi, Giuseppe Mingione, and Kaj Nystr\"om.
\newblock Sharp regularity for evolutionary obstacle problems, interpolative
  geometries and removable sets.
\newblock {\em J. Math. Pures Appl., to appear.}, 2012.

\bibitem{LS03}
Ki-Ahm Lee and Henrik Shahgholian.
\newblock Hausdorff measure and stability for the {$p$}-obstacle problem
  {$(2<p<\infty)$}.
\newblock {\em J. Differential Equations}, 195(1):14--24, 2003.

\bibitem{LM13a}
Erik Lindgren and R\'egis Monneau.
\newblock Pointwise estimates for the heat equation. {A}pplication to the free
  boundary of the obstacle problem with {D}ini coefficients.
\newblock {\em Indiana Univ. Math. J.}, 62:171--199, 2013.

\bibitem{LM13b}
Erik {Lindgren} and R\'egis {Monneau}.
\newblock Pointwise regularity of the free boundary for the parabolic obstacle
  problem.
\newblock {\em ArXiv: 1305.7349}, 2013.

\bibitem{Lin88}
Peter Lindqvist.
\newblock Regularity for the gradient of the solution to a nonlinear obstacle
  problem with degenerate ellipticity.
\newblock {\em Nonlinear Anal.}, 12(11):1245--1255, 1988.

\bibitem{Lin12}
Peter Lindqvist.
\newblock On the time derivative in an obstacle problem.
\newblock {\em Rev. Mat. Iberoam.}, 28(2):577--590, 2012.

\bibitem{LP12}
Peter Lindqvist and Mikko Parviainen.
\newblock Irregular time dependent obstacles.
\newblock {\em J. Funct. Anal.}, 263(8):2458--2482, 2012.

\bibitem{MZ97}
Jan Mal{\'y} and William~P. Ziemer.
\newblock {\em Fine regularity of solutions of elliptic partial differential
  equations}, volume~51 of {\em Mathematical Surveys and Monographs}.
\newblock American Mathematical Society, Providence, RI, 1997.

\bibitem{Mon03}
R{\'e}gis Monneau.
\newblock On the number of singularities for the obstacle problem in two
  dimensions.
\newblock {\em J. Geom. Anal.}, 13(2):359--389, 2003.

\bibitem{Nor86}
Tullia Norando.
\newblock {$C^{1,\alpha}$} local regularity for a class of quasilinear elliptic
  variational inequalities.
\newblock {\em Boll. Un. Mat. Ital. C (6)}, 5(1):281--292 (1987), 1986.

\bibitem{PS07}
Arshak Petrosyan and Henrik Shahgholian.
\newblock Parabolic obstacle problems applied to finance.
\newblock In {\em Recent developments in nonlinear partial differential
  equations}, volume 439 of {\em Contemp. Math.}, pages 117--133. Amer. Math.
  Soc., Providence, RI, 2007.

\bibitem{Sha03}
Henrik Shahgholian.
\newblock Analysis of the free boundary for the {$p$}-parabolic variational
  problem {$(p\geq 2)$}.
\newblock {\em Rev. Mat. Iberoamericana}, 19(3):797--812, 2003.

\bibitem{Zaj87}
Lud{\v{e}}k Zaj{\'{\i}}{\v{c}}ek.
\newblock Porosity and {$\sigma$}-porosity.
\newblock {\em Real Anal. Exchange}, 13(2):314--350, 1987/88.

\end{thebibliography}

\end{document}